\documentclass[a4paper,11pt]{article}

\usepackage{amsmath}
\usepackage{amssymb}
\usepackage{amsthm}
\usepackage{amsfonts}
\usepackage{appendix}
\usepackage{mathtools}
\usepackage{ascmac}
\usepackage{spalign}
\usepackage{fancybox}
\usepackage[mathscr]{eucal}
\usepackage{footnpag}
\usepackage{siunitx}
\usepackage{multicol}
\usepackage{bbm}
\usepackage{framed}
\usepackage{here}

\usepackage{wrapfig}
\usepackage{float}

\theoremstyle{definition}
\newtheorem{theorem}{Theorem}[section]
\newtheorem*{theorem*}{Theorem}

\newtheorem*{definition*}{Definition}
\newtheorem{proposition}{Proposition}[section]
\newtheorem*{proposition*}{Proposition}
\newtheorem{lemma}{Lemma}[section]
\newtheorem*{lemma*}{Lemma}
\newtheorem{corollary}{Corollary}[section]
\newtheorem*{corollary*}{Corollary}
\newtheorem{remark}{Remark}[section]
\newtheorem*{remark*}{Remark}

\newtheorem*{example*}{Example}

\theoremstyle{example}

\newfont{\bg}{cmr9 scaled\magstep4}
\newcommand{\bigzero}{
\smash{\lower1.0ex\hbox{\bg 0}}}

\newcommand{\ctext}[1]{\raise0.2ex\hbox{\textcircled{\scriptsize{#1}}}}

\DeclareMathOperator{\Kf}{Kf}

\title{Effective Resistance and Generalized Bejaia--Pisa Sequences
on Complete Graphs with Circulant Distance Deletions}
\author{Shunya Tamura\thanks{Okegawa City Okegawa West Junior High School, Saitama, 363-0027, Japan, e-mail: shunya.tamura059@gmail.com}, 
Yuuho Tanaka \thanks{Faculty of Science and Technology, Oita University, Oita, 870-1192, Japan, e-mail: tanaka-yuuho@oita-u.ac.jp }
}

\date{}
\begin{document}
\maketitle

\begin{abstract}

In this paper, we investigate the effective resistance on the graph $G_N^{(r)}$, which is obtained by deleting all edges corresponding to circular distances $\{\pm1, \pm2, \dots, \pm r\}$ from the complete graph $K_N$.
We utilize the cyclic symmetry of the graph to diagonalize the Laplacian matrix via the discrete Fourier basis and derive a finite trigonometric sum representation for the effective resistance between two vertices at distance $\ell$.

Specifically, we treat the cases $r=1$ and $r=2$ in detail and provide explicit formulas.
For the case of $r=1$, we use Fourier analysis to rederive the closed form in terms of Bejaia and Pisa numbers given by Chair.
For the case of $r=2$, we show that the denominator reduces to a quadratic polynomial with complex roots and introduce a generalized Bejaia--Pisa-type complex sequence.
Using this sequence, we provide some closed forms for the effective resistance and various related formulas.

\end{abstract}



\section{Introduction}

The effective resistance of a graph is a basic quantity in electrical circuit theory, random walks, stochastic processes, and spectral graph theory and is defined via the Moore--Penrose pseudoinverse of the Laplacian matrix. 
For highly symmetric graphs in particular, it is known that the effective resistance is deeply related to specific sequences and trigonometric sums \cite{Nash-Williams1959,Tetali1991,Wu2004}.

Chair \cite{Chair} analyzed the effective resistance and the average hitting time of the graph obtained by deleting $N$ edges from the complete graph and showed that their formulas can be described by Fibonacci-type sequences called Bejaia and Pisa numbers. 
Miezaki and Tamura \cite{MT2026} provided explicit formulas for the average hitting time of the $k$-th power graph of a cycle by using sequences that satisfy a three-term recurrence relation via discrete Fourier analysis.
Moreover, Tamura \cite{Tamura2026} provided explicit formulas for the effective resistance and average hitting time of graphs obtained by removing edges at distance $\pm r$ from the complete graph. 
These results serve as representative examples demonstrating that a rich structure emerges between effective resistances and integer sequences in graphs obtained by periodically deleting edges from the complete graph.

In this paper, we consider the graph $G_N^{(r)}$, which is obtained by deleting all edges corresponding to the circular distances $\{\pm1,\pm2,\dots,\pm r\}$ from the complete graph $K_N$.
Specifically, the case of $r=1$ represents the graph 
\[
G_N^{(1)} = K_N \setminus C_N
\]
obtained by removing the edges of the cycle $C_N$ from the complete graph $K_N$. 
For the odd-vertex case, this essentially coincides with the graph studied by Chair \cite{Chair}. 
Therefore, we use Fourier analysis to rederive the closed formula given by Chair \cite{Chair} for the case of $r=1$ and derive a new closed form for the case of $r=2$.

The Laplacian matrix of the graph $G_N^{(r)}$ can be diagonalized via the discrete Fourier basis due to its cyclic symmetry.
We use this diagonalization to obtain the formula for the effective resistance between two vertices at distance $\ell$ for any $r$.

For the case of $r=1$, we rederive the trigonometric sum representation given by Chair from this general formula and show that the closed form in terms of Bejaia and Pisa numbers emerges as a natural consequence of Fourier and complex analysis.

Furthermore, as the main result of this paper, we introduce new Fibonacci-type sequences corresponding to the case of $r=2$ and provide closed forms for the effective resistance and average hitting time.

For $r=1$, the characteristic equation has real roots, so the Bejaia and Pisa numbers appear as real sequences. 
However, for $r=2$, the characteristic equation has complex roots, and the corresponding sequences are naturally extended to complex sequences. 
Consequently, the effective resistance for $r=2$ is described using these generalized Bejaia--Pisa-type sequences. 
This provides a new sequence structure that emerges when multiple distance classes are removed from the complete graph.

\begin{theorem}[Case: $r=2$]\label{thm:r2-chair}
Let $N\ge 7$ be an odd number and $r$ be an integer such that $1\le \ell\le \frac N2-1$.

Then, the effective resistance between two vertices $0$ and $\ell$ of the graph $G_N^{(2)}$ is
\[
R_{N,2}(\ell)
=
\frac{4}{\sqrt{4N-25}}
\Im\left(
C_N
-
\frac{C_N}{2}\mathcal P_\ell^{(2)}
-
\frac12\mathcal B_\ell^{(2)}
\right),
\]
where $\Im(\cdot)$ denotes the imaginary part, $\mathcal P_\ell^{(2)}$ is the $\ell$-th Pisa number for $r=2$, and $\mathcal B_\ell^{(2)}$ is the $\ell$-th Bejaia number for $r=2$. 
Furthermore, $C_N$ is defined as
\[
C_N
:=
\frac{
2+\mathcal P_N^{(2)}+d\mathcal B_N^{(2)}
}{
d\left(
2-\mathcal P_N^{(2)}-d\mathcal B_N^{(2)}
\right)
}.
\]
\end{theorem}

The formula for the effective resistance obtained in this paper immediately yields explicit formulas for the Kirchhoff index, the number of spanning trees, and the average hitting time of a simple random walk.

The organization of this paper is as follows.
In Section \ref{sec02}, we give the definition and basic properties of the graph $G_N^{(r)}$.
In Section \ref{sec03}, we use the diagonalization of the Laplacian matrix via the discrete Fourier basis to derive finite sum representations for the effective resistance and average hitting time. 
In Section \ref{sec04}, we consider the cases of $r=1$ and $r=2$, deriving closed forms for the effective resistance, average hitting time, Kirchhoff index, and the number of spanning trees. 
For $r=1$, we give closed forms in terms of Bejaia and Pisa numbers and assess the correspondence with Chair's results. 
For $r=2$, we introduce generalized Bejaia--Pisa-type sequences and derive the closed forms using them. 
Finally, we provide concluding remarks.

\section{Preliminary} \label{sec02}

In this section, we present the definition of the graph considered in this paper along with its distance structure and symmetry, which are necessary for calculating the effective resistance and various byproducts later. 
Throughout the remainder of this paper, we assume that the number of vertices $N\ge 7$ is odd.

Let $V=\mathbb{Z}/N\mathbb{Z}$ be the vertex set. 
Assume that $r$ is an integer satisfying $1 \leq r \leq (N-3)/2$. 
We define the graph $G_N^{(r)}$ as the one obtained by deleting all edges corresponding to the circular distances $\{\pm1,\pm2,\dots,\pm r\}$ from the complete graph $K_N$. 
That is, two distinct vertices $x, y \in V$ are adjacent in $G_N^{(r)}$ if and only if their circular distance satisfies
\[
d(x,y)\in\{r+1,r+2,\dots,(N-1)/2\}.
\]
The condition $r \leq (N-3)/2$ is required to ensure that $G_N^{(r)}$ is a non-trivial connected graph.
Indeed, if $r=(N-1)/2$, all edges between distinct vertices would be deleted, yielding an edgeless discrete graph.
According to its definition, $G_N^{(r)}$ is a regular graph with each vertex having a constant degree of $\deg(G_N^{(r)})=N-1-2r$.
For the case of $r=1$ in particular, the graph becomes $G_N^{(1)} = K_N \setminus C_N$, which coincides with the graph studied by Chair \cite{Chair}. 
Moreover, for the case of $r=2$, it represents the graph obtained by removing the edges corresponding to the circular distances $\pm1$ and $\pm2$ from the complete graph.

In this paper, we focus particularly on the Fibonacci-type sequence structures arising for $r=1$ and $r=2$.
For vertices $x,y\in V$, the circular distance is defined as
\[
d(x,y)=\min\{|x-y|,\,N-|x-y|\}.
\]
Since $N$ is odd, the maximum distance is $q=(N-1)/2$, so $d(x,y)\in\{0,1,2,\dots,q\}$ holds.
We fix the vertex $0\in V$ as the base point. 
For each distance $k=0,1,\dots,q$, the set $D_k=\{x\in V\mid d(0,x)=k\}$ is called a distance class. 
Since there exist two distinct vertices $x=\pm k \pmod N$ for each distance $k\ge1$, the size of each distance class is
\[
|D_0|=1,\qquad
|D_k|=2
\quad(1\le k\le q).
\]
The graph $G_N^{(r)}$ is invariant under the translation action $x\mapsto x+a$ and the inversion map $x\mapsto -x$ of the cyclic group $\mathbb{Z}/N\mathbb{Z}$. 
Therefore, $G_N^{(r)}$ is a vertex-transitive circulant graph.
Due to this high symmetry, the effective resistance and the average hitting time depend only on the circular distance between two vertices, rather than on the vertices themselves.
Specifically, for any $x,y\in V$, we have $R(x,y)=R(0,d(x,y))$.
Therefore, in what follows, we denote the effective resistance between two vertices at distance $\ell$ by $R(0,\ell)$.

Furthermore, due to the cyclic symmetry of $G_N^{(r)}$, its Laplacian matrix can be diagonalized via the discrete Fourier basis. 
Therefore, we can express the effective resistance as a finite trigonometric sum and subsequently derive the Fibonacci-type sequence structures akin to Bejaia and Pisa numbers for $r=1$ and $r=2$.

\begin{remark}
If $N$ is even, the maximum distance is $N/2$, and the distance class corresponding to $N/2$ consists of only a single vertex.
Therefore, a self-symmetric term appears when simplifying the sum using the symmetry of $m \leftrightarrow N-m$.

We provide supplementary details for the case of even $N$ in Appendix \ref{app:evenN}.
\end{remark}

\section{General Representation for Effective Resistance} \label{sec03}

In this section, we diagonalize the Laplacian matrix using the discrete Fourier basis, derive a finite-sum representation of the effective resistance, and obtain Theorem \ref{mainthm1}.

\begin{lemma}\label{lem31}

We identify the vertex set with $V=\mathbb{Z}/N\mathbb{Z}$ and let $\mathbb{C}^N$ denote the space of all complex-valued functions on $V$.
Let $\omega:=e^{2\pi i/N}$.
For each $m=0,1,\dots,N-1$, we define
\[
\varphi_m(x):=\frac{1}{\sqrt{N}}\omega^{mx} \quad (x\in V).
\]
Then, $\{\varphi_m\}_{m=0}^{N-1}$ forms an orthonormal basis for $\mathbb{C}^N$ with respect to the standard inner product
\[
\langle f,g\rangle=\sum_{x\in V} f(x)\overline{g(x)}.
\]
\end{lemma}

\begin{proof}
For $m,m'=0,1,\dots,N-1$, we have
\[
\langle \varphi_m,\varphi_{m'}\rangle
=
\frac{1}{N}\sum_{x=0}^{N-1}\omega^{(m-m')x}.
\]

If $m=m'$, the right-hand side evaluates to $1$.
Conversely, if $m\ne m'$, then $\omega^{m-m'}\ne1$, yielding
\[
\sum_{x=0}^{N-1}\omega^{(m-m')x}
=
\frac{1-\omega^{(m-m')N}}{1-\omega^{m-m'}}
=0.
\]
Therefore, $\{\varphi_m\}_{m=0}^{N-1}$ is an orthonormal system.
Since the number of these vectors equals $\dim \mathbb{C}^N=N$, they form an orthonormal basis.
\end{proof}

\begin{proposition}\label{prop32}
Let $L$ be the Laplacian matrix of the graph $G_N^{(r)}$.
Then, we have 
\[
L\varphi_m=\lambda_m\varphi_m
\]
for $0\le m\le N-1$.
Moreover, the eigenvalues are given by
\[
\lambda_m=
\begin{cases}
0, & m=0,\\[4pt]
(N-2r)+2\displaystyle\sum_{j=1}^{r}
\cos\left(\displaystyle\frac{2\pi mj}{N}\right),
& 1\le m\le N-1
\end{cases}.
\]
\end{proposition}

\begin{proof}
We define the cyclic shift operator $T:\mathbb{C}^N\to\mathbb{C}^N$ by
\[
(Tu)(x):=u(x+1)
\qquad (x\in \mathbb{Z}/N\mathbb{Z}).
\]
Since the adjacency relation of $G_N^{(r)}$ depends only on the difference $y-x$, the Laplacian matrix $L$ commutes with the cyclic shift operator $T$. 
That is, $LT=TL$. 
Moreover, from
\[
(T\varphi_m)(x)
=\varphi_m(x+1)
=\frac{1}{\sqrt{N}}\omega^{m(x+1)}
=\omega^m\varphi_m(x)
\]
we obtain $T\varphi_m=\omega^m\varphi_m$.
Therefore, we have
\[
T(L\varphi_m)=L(T\varphi_m)=\omega^m L\varphi_m.
\]
Thus, $L\varphi_m$ belongs to the eigenspace of $T$ associated with the eigenvalue $\omega^m$. 
Since this eigenspace is one-dimensional and coincides with $\mathbb{C}\varphi_m$, there exists a scalar $\lambda_m$ such that $L\varphi_m=\lambda_m\varphi_m$.

For any function $u:V\to\mathbb{C}$,
\[
(Lu)(x)=\deg(x)u(x)-\sum_{y\sim x}u(y).
\]
In $G_N^{(r)}$, since the edges corresponding to the distances $\pm1,\dots,\pm r$ are deleted, the neighbors of a vertex $x$ are given by
\[
x\pm(r+1),\ x\pm(r+2),\dots,\ x\pm q,
\qquad q=\frac{N-1}{2}.
\]
Consequently, the degree is
\[
\deg(x)=2(q-r)=N-1-2r,
\]
yielding
\begin{equation}\label{eq:lap-action}
(Lu)(x)
=(N-1-2r)u(x)-\sum_{k=r+1}^{q}\{u(x+k)+u(x-k)\}.
\end{equation}

Let $u=\varphi_m$.
Then, since $\varphi_m(x\pm k)=\omega^{\pm mk}\varphi_m(x)$ holds,
\[
\varphi_m(x+k)+\varphi_m(x-k)
=
2\cos\left(\frac{2\pi mk}{N}\right)\varphi_m(x)
\]
holds.
Thus,
\[
(L\varphi_m)(x)
=
\left(
N-1-2r
-
2\sum_{k=r+1}^{q}
\cos\left(\frac{2\pi mk}{N}\right)
\right)\varphi_m(x).
\]
Therefore,
\begin{equation}\label{eq:eigen-before}
\lambda_m
=
N-1-2r
-
2\sum_{k=r+1}^{q}
\cos\left(\frac{2\pi mk}{N}\right).
\end{equation}

For $m\ne 0$, we have
\[
\sum_{x=0}^{N-1}\omega^{mx}=0.
\]
Taking the real part, we obtain
\[
1+2\sum_{k=1}^{q}
\cos\left(\frac{2\pi mk}{N}\right)=0.
\]
Therefore,
\[
\sum_{k=r+1}^{q}
\cos\left(\frac{2\pi mk}{N}\right)
=
-\frac12
-
\sum_{k=1}^{r}
\cos\left(\frac{2\pi mk}{N}\right).
\]
Substituting this into \eqref{eq:eigen-before}, we obtain
\[
\lambda_m
=
(N-2r)
+
2\sum_{k=1}^{r}
\cos\left(\frac{2\pi mk}{N}\right).
\]

For $m=0$, since $\varphi_0$ is a constant vector, we have $L\varphi_0=0$.
Thus, $\lambda_0=0$.
\end{proof}

\begin{lemma}\label{lem33}
Let $L^+$ be the Moore--Penrose pseudoinverse of the Laplacian matrix $L$.
Then,
\[
L^+
=
\sum_{m=1}^{N-1}
\frac{1}{\lambda_m}\,
\varphi_m\varphi_m^{*}.
\]
\end{lemma}

\begin{proof}
From Lemma \ref{lem31} and Proposition \ref{prop32}, we have
\[
L=\sum_{m=0}^{N-1}\lambda_m\varphi_m\varphi_m^{*}.
\]
The Moore--Penrose inverse is obtained by replacing each non-zero eigenvalue with its reciprocal, omitting the component corresponding to the zero eigenvalue.
Thus, it follows that
\[
L^+
=\sum_{m=1}^{N-1}\frac{1}{\lambda_m}\,\varphi_m\varphi_m^{*}.
\]
\end{proof}

Next, we introduce the effective resistance.
Let $G=(V,E)$ be a connected undirected graph, 
with $L$ being its Laplacian matrix and $L^+$ its Moore--Penrose inverse.
Let $e_x$ denote the standard basis vector corresponding to a vertex $x\in V$. 
Then, for any vertices $x,y\in V$, the effective resistance is defined as
\[
R(x,y)
:=
(e_x-e_y)^*L^+(e_x-e_y).
\]

\begin{proposition}\label{prop31}
The effective resistance between two vertices $0$ and $\ell$ at distance $\ell$ of the graph $G_N^{(r)}$ is given by
\[
R(0,\ell)
=
\frac{2}{N}
\sum_{m=1}^{N-1}
\frac{
1-\cos\left(\frac{2\pi m\ell}{N}\right)
}{
\lambda_m
}.
\]
\end{proposition}

\begin{proof}
From lemma \ref{lem33} and the definition of the effective resistance,
\begin{align*}
R(x,y)
&=
(e_x-e_y)^*L^+(e_x-e_y)\\
&=
\sum_{m=1}^{N-1}
\frac{1}{\lambda_m}
\left|
(e_x-e_y)^*\varphi_m
\right|^2\\
&=
\sum_{m=1}^{N-1}
\frac{
|\varphi_m(x)-\varphi_m(y)|^2
}{
\lambda_m
}.
\end{align*}
This takes the same form as the spectral representation given by Wu \cite{Wu2004}.

Setting $x=0$ and $y=\ell$, we have
\begin{align*}
|\varphi_m(0)-\varphi_m(\ell)|^2
&=
\left|
\frac{1}{\sqrt{N}}(1-\omega^{m\ell})
\right|^2\\
&=
\frac{1}{N}
(1-\omega^{m\ell})(1-\omega^{-m\ell})\\
&=
\frac{2}{N}
\left(
1-\cos\frac{2\pi m\ell}{N}
\right).
\end{align*}
Therefore, we obtain
\[
R(0,\ell)
=
\frac{2}{N}
\sum_{m=1}^{N-1}
\frac{
1-\cos\left(\frac{2\pi m\ell}{N}\right)
}{
\lambda_m
}.
\]
\end{proof}

From the above, we obtain the following theorem.

\begin{theorem}\label{mainthm1}
Let $r$ be an integer such that $1\le r\le (N-3)/2$.
Then, the effective resistance between two vertices $0$ and $\ell$ at distance $\ell\in\{1,2,\dots,(N-1)/2\}$ of the graph $G_N^{(r)}$ is given by
\[
R(0,\ell)
=
\frac{4}{N}
\sum_{m=1}^{(N-1)/2}
\frac{
1-\cos\left(\frac{2\pi m\ell}{N}\right)
}{
(N-2r)+2\displaystyle\sum_{j=1}^{r}
\cos\left(\frac{2\pi mj}{N}\right)
}.
\]
\end{theorem}

\begin{proof}
From Proposition \ref{prop31}, we have
\[
R(0,\ell)
=
\frac{2}{N}
\sum_{m=1}^{N-1}
\frac{
1-\cos\left(\frac{2\pi m\ell}{N}\right)
}{
\lambda_m
}.
\]
Since Proposition \ref{prop32} provides
\[
\lambda_m
=
(N-2r)+2\sum_{j=1}^{r}
\cos\left(\frac{2\pi mj}{N}\right)
\]
for $1\le m\le N-1$, we can write
\[
R(0,\ell)
=
\frac{2}{N}
\sum_{m=1}^{N-1}
f(m),
\]
where
\[
f(m)
=
\frac{
1-\cos\left(\frac{2\pi m\ell}{N}\right)
}{
(N-2r)+2\displaystyle\sum_{j=1}^{r}
\cos\left(\frac{2\pi mj}{N}\right)
}.
\]

By the symmetry of the cosine function, it holds that $f(N-m)=f(m)$.
Moreover, since $N$ is odd, $m$ and $N-m$ form distinct pairs for $m=1,2,\dots,(N-1)/2$. 
Therefore, we have
\[
\sum_{m=1}^{N-1}f(m)
=
2\sum_{m=1}^{(N-1)/2}f(m).
\]
This yields
\[
R(0,\ell)
=
\frac{4}{N}
\sum_{m=1}^{(N-1)/2}
\frac{
1-\cos\left(\frac{2\pi m\ell}{N}\right)
}{
(N-2r)+2\displaystyle\sum_{j=1}^{r}
\cos\left(\frac{2\pi mj}{N}\right)
}.
\]
\end{proof}

\subsection{Average Hitting Time}

Let us consider a simple random walk on $G_N^{(r)}$.
This is a stochastic process in which at each time step the walker moves from the current position to an adjacent vertex with uniform probability.
For vertices $x,y\in V$, let $T_y$ be the hitting time to vertex $y$.
We define the average hitting time starting from vertex $x$ to vertex $y$ as $h(x, y)=\mathbb{E}_x[T_y]$.

Nash-Williams \cite{Nash-Williams1959} provided the following relationship between the average hitting time and the effective resistance.
\begin{theorem}[Nash-Williams \cite{Nash-Williams1959}] \label{thm:NW}
\[
h(x, y)+h(y, x)=2|E|R(x,y). 
\]
\end{theorem}

The number of edges in $G_N^{(r)}$ is $|E|=N(N-1-2r)/2$, and since the graph is vertex-transitive, we have $h(x, y)=h(y, x)$.
Therefore, by Theorem \ref{thm:NW}, we obtain
\[
h(x, y)=|E|R(x,y),
\]
yielding the following closed formula for the average hitting time.

\begin{corollary}\label{cor:hitting-general}
Let $r$ be an integer such that $1\le r\le (N-3)/2$.
With the two vertices $0$ and $\ell$ at distance $\ell$ of the graph $G_N^{(r)}$, it follows that
\[
h(0, \ell)=2(N-1-2r)\sum_{m=1}^{(N-1)/2}\frac{1-\cos\left(\frac{2\pi m\ell}{N}\right)}{(N-2r)+2\displaystyle\sum_{j=1}^{r}\cos\left(\frac{2\pi mj}{N}\right)}.
\]
\end{corollary}

For $r=1,2$, we can also obtain closed forms expressed in terms of Fibonacci-type sequences and generalized sequences.
We discuss these details in the next section.

\section{Sequence Structures That Appear When $r=1$ and $r=2$} \label{sec04}

In this section, we analyze Theorem \ref{mainthm1} for $r=1$ and $r=2$.
For $r=1$, the graph $G_N^{(1)}=K_N\setminus C_N$ coincides with the one studied by Chair \cite{Chair}, and we rederive the closed forms expressed in terms of Bejaia and Pisa numbers. 
On the other hand, for $r=2$, the denominator reduces to a quadratic polynomial, and generalized Bejaia--Pisa-type sequences, which are naturally determined from its complex roots, arise.

Furthermore, we provide explicit formulas for the Kirchhoff index, the number of spanning trees, and the average hitting time for each case.

The Kirchhoff index $\Kf(G)$ of a graph $G=(V,E)$ is defined in terms of the effective resistance as follows \cite{KR1993,LNT1999}.
\[
\Kf(G):=\frac{1}{2}\sum_{x,y\in V}R(x,y).
\]

Next, the graph complexity of $G$, denoted by $t(G)$, refers to the number of spanning trees of $G$.
The number of spanning trees can be obtained using the Matrix Tree Theorem.

\begin{theorem}[Matrix Tree Theorem, Kirchhoff \cite{Kir}] \label{thm:MT}
Let $\lambda_i$ ($0\leq i\leq N-1$) be the eigenvalues of the Laplacian matrix of an undirected graph $G$.
Then,
\[
t(G)=\frac{1}{N}\prod_{i=1}^{N-1}\lambda_i.
\]
\end{theorem}
From Proposition \ref{prop32} and Theorem \ref{thm:MT}, we obtain
\begin{equation} \label{eq:mt}
t(G_N^{(r)})=\frac{1}{N}\prod_{i=1}^{N-1}\left( (N-2r)+2\sum_{j=1}^{r}
\cos\left(\frac{2\pi ij}{N}\right) \right).
\end{equation}

We denote by $t(G;x,y)$ the graph complexity of $G$ obtained by identifying two vertices $x$ and $y$.
The following theorem also shows that the effective resistance and the graph complexity are closely related.

\begin{theorem}[Kirchhoff \cite{Kir}] \label{Kir}
\[
t(G;x,y)=R(x,y) \cdot t(G). 
\]
\end{theorem}

\subsection{Case: $r=1$}

Let $N$ be an odd number such that $N \ge 5$.
We denote the effective resistance between two vertices at distance $\ell$ of $G_N^{(1)}$ by $R_{N,1}(\ell)$.

Let
\[
\Delta=N(N-4),
\qquad
\alpha=\frac{N-2+\sqrt{\Delta}}{2},
\qquad
\beta=\frac{N-2-\sqrt{\Delta}}{2}.
\]
Then, we have
\[
\alpha+\beta=N-2,
\qquad
\alpha\beta=1.
\]
Furthermore,
\[
\mathcal B_k(N)=\frac{\alpha^k-\beta^k}{\sqrt{\Delta}}
\]
is called the $k$-th Bejaia number, and
\[
\mathcal P_k(N)=\alpha^k+\beta^k
\]
is called the $k$-th Pisa number.

From the above, we obtain the following theorem.

\begin{theorem}[Chair \cite{Chair}]\label{thm:r1}
Let $N$ be an odd number such that $N \ge 5$ and an integer $\ell$ such that $1\le \ell\le (N-1)/2$.
Then, the effective resistance between two vertices at distance $\ell$ is
\[
R_{N,1}(\ell)
=
(-1)^\ell \mathcal B_\ell(N)
+
\frac{1-\beta^N}{\sqrt{\Delta}(1+\beta^N)}
\left(
2-(-1)^\ell\mathcal P_\ell(N)
\right),
\]
where $\mathcal P_\ell(N)$ is the $\ell$-th Pisa number, and $\mathcal B_\ell(N)$ is the $\ell$-th Bejaia number.
\end{theorem}

\begin{proof}
Substituting $r=1$ into Theorem \ref{mainthm1} yields
\[
R_{N,1}(\ell)
=
\frac{4}{N}
\sum_{m=1}^{(N-1)/2}
\frac{1-\cos\left(\frac{2\pi m\ell}{N}\right)}
{(N-2)+2\cos\left(\frac{2\pi m}{N}\right)}.
\]

Let $\theta_m=2\pi m/N$.
Since the summand is invariant under replacing $m$ with $N-m$, and the term for $m=0$ is $1-\cos(\ell\theta_0)=0$, we can write
\[
R_{N,1}(\ell)
=
\frac{2}{N}
\sum_{m=0}^{N-1}
\frac{1-\cos(\ell\theta_m)}
{(N-2)+2\cos\theta_m}.
\]

Let $z=e^{i\theta}$, then we have
\[
(N-2)+2\cos\theta
=
(N-2)+z+z^{-1}.
\]
On the other hand, we have
\[
\alpha(1+\beta z)(1+\beta z^{-1})
=
(\alpha+\beta)+(z+z^{-1})
=
(N-2)+z+z^{-1}.
\]
Therefore,
\[
\frac{1}{(N-2)+2\cos\theta}
=
\frac{1}{\alpha(1+\beta e^{i\theta})(1+\beta e^{-i\theta})}.
\]
Furthermore, since $|\beta|<1$, we have
\[
\frac{1}{1+\beta e^{i\theta}}
=
\sum_{p=0}^{\infty}(-\beta)^p e^{ip\theta},
\qquad
\frac{1}{1+\beta e^{-i\theta}}
=
\sum_{q=0}^{\infty}(-\beta)^q e^{-iq\theta}.
\]
Grouping the terms by the index difference $n=p-q$, we obtain
\[
\frac{1}{(N-2)+2\cos\theta}
=
\frac{1}{\sqrt{\Delta}}
\sum_{n\in\mathbb Z}(-\beta)^{|n|}e^{in\theta}.
\]

Letting
\[
A_\ell
:=
\frac1N
\sum_{m=0}^{N-1}
\frac{\cos(\ell\theta_m)}
{(N-2)+2\cos\theta_m},
\]
we have
\[
R_{N,1}(\ell)=2(A_0-A_\ell).
\]
Using the Fourier expansion above, we obtain
\[
A_\ell
=
\frac{1}{\sqrt{\Delta}}
\sum_{t\in\mathbb Z}
(-\beta)^{|tN+\ell|}.
\]
In particular, for $1\le \ell\le (N-1)/2$, we have
\[
A_\ell
=
\frac{(-1)^\ell}{\sqrt{\Delta}}
\frac{\beta^\ell-\beta^{N-\ell}}{1+\beta^N},
\]
and
\[
A_0
=
\frac{1}{\sqrt{\Delta}}
\frac{1-\beta^N}{1+\beta^N}.
\]
Therefore,
\[
R_{N,1}(\ell)
=
\frac{2}{\sqrt{\Delta}}
\left(
\frac{1-\beta^N}{1+\beta^N}
-
(-1)^\ell
\frac{\beta^\ell-\beta^{N-\ell}}{1+\beta^N}
\right).
\]

Finally, we rewrite this using
\[
\mathcal B_\ell(N)=\frac{\beta^{-\ell}-\beta^\ell}{\sqrt{\Delta}},
\qquad
\mathcal P_\ell(N)=\beta^{-\ell}+\beta^\ell.
\]
Substituting the identity
\[
2(\beta^\ell-\beta^{N-\ell})
=
(1-\beta^N)\mathcal P_\ell(N)
-
(1+\beta^N)\sqrt{\Delta}\,\mathcal B_\ell(N),
\]
we obtain
\[
R_{N,1}(\ell)
=
(-1)^\ell \mathcal B_\ell(N)
+
\frac{1-\beta^N}{\sqrt{\Delta}(1+\beta^N)}
\left(
2-(-1)^\ell\mathcal P_\ell(N)
\right).
\]
\end{proof}

Furthermore, from Theorem \ref{thm:NW} and Theorem \ref{thm:r1},
we can give the closed formula for the average hitting time of $G_N^{(1)}$.

\begin{corollary}[Chair \cite{Chair}]
\[
h(0, \ell)
=\frac{N(N-3)}{2}\left(
(-1)^\ell \mathcal B_\ell(N) + \frac{1-\beta^N}{\sqrt{\Delta}(1+\beta^N)}
\left(
2-(-1)^\ell\mathcal P_\ell(N)
\right)
\right).
\]
\end{corollary}

\begin{proof}

\begin{align*}
h(0, \ell)
&=|E|R_{N,1}(\ell) \\
&=\frac{N(N-3)}{2}\left((-1)^\ell \mathcal B_\ell(N) + \frac{1-\beta^N}{\sqrt{\Delta}(1+\beta^N)}\left(2-(-1)^\ell\mathcal P_\ell(N)
\right)\right).
\end{align*}

\end{proof}

\subsubsection{Kirchhoff Index}

Before calculating the Kirchhoff index, we prepare some formulas.

\begin{lemma} \label{lem:PB41}

\begin{enumerate}

    \item\label{formula:ab} $(1+\alpha)(1+\beta)=N$ 
    
    \item\label{formula:1} $\sum_{i=1}^{N-1}(-1)^ii\alpha^i=\frac{1}{1+\alpha}\left( \sum_{i=1}^N(-\alpha)^i-N(-\alpha)^N \right)$

    \item\label{formula:2} $\sum_{i=1}^{N-1}(-1)^ii\beta^i=\frac{1}{1+\beta}\left( \sum_{i=1}^N(-\beta)^i-N(-\beta)^N \right)$

    \item\label{formula:3} $\sum_{i=1}^{N-1}(-1)^i\mathcal B_i(N)=\frac{(-1)^{N+1}(\mathcal B_N(N)+\mathcal B_{N-1}(N))-1}{N}$

    \item\label{formula:4} $\sum_{i=1}^{N-1}(-1)^i\mathcal P_i(N)=\frac{(-1)^{N+1}(\mathcal P_N(N)+\mathcal P_{N-1}(N))-N}{N}$

    \item\label{formula:5} $\sum_{i=1}^{N-1}(-1)^ii\mathcal B_i(N)=\frac{(-1)^N(\mathcal B_N(N)-N(\mathcal B_N(N)+\mathcal B_{N-1}(N)))}{N}$

    \item\label{formula:6} $\sum_{i=1}^{N-1}(-1)^ii\mathcal P_i(N)=\frac{(-1)^N(\mathcal P_N(N)-N(\mathcal P_N(N)+\mathcal P_{N-1}(N)))-2}{N}$.
\end{enumerate}

\end{lemma}

\begin{proof}

\begin{itemize}

\item[(\ref{formula:ab})]
\[
(1+\alpha)(1+\beta)=\alpha\beta+\alpha+\beta+1=1+N-2+1=N.
\]

\item[(\ref{formula:1})]

\begin{align*}
\sum_{i=1}^{N-1}(-1)^ii\alpha^i+\alpha\sum_{i=1}^{N-1}(-1)^ii\alpha^i
&=-\alpha+\alpha^2+\cdots+(-1)^{N-1}\alpha^{N-1}+(-1)^{N-1}(N-1)\alpha^N \\
&=\sum_{i=1}^N(-\alpha)^i-N(-\alpha)^N.
\end{align*}
From the above, we obtain
\[
\sum_{i=1}^{N-1}(-1)^ii\alpha^i=\frac{1}{1+\alpha}\left( \sum_{i=1}^N(-\alpha)^i-N(-\alpha)^N \right).
\]

\item[(\ref{formula:2})]
\begin{align*}
\sum_{i=1}^{N-1}(-1)^ii\beta^i+\beta\sum_{i=1}^{N-1}(-1)^ii\beta^i
&=-\beta+\beta^2+\cdots+(-1)^{N-1}\beta^{N-1}+(-1)^{N-1}(N-1)\beta^N \\
&=\sum_{i=1}^N(-\beta)^i-N(-\beta)^N.
\end{align*}
From the above, we obtain
\[
\sum_{i=1}^{N-1}(-1)^ii\beta^i=\frac{1}{1+\beta}\left( \sum_{i=1}^N(-\beta)^i-N(-\beta)^N \right).
\]

\item[(\ref{formula:3})]
\begin{align*}
&\sum_{i=1}^{N-1}(-1)^i\mathcal B_i(N) \\
&=\frac{1}{\sqrt{\Delta}}\sum_{i=1}^{N-1}(-1)^i(\alpha^i-\beta^i) \\
&=\frac{1}{\sqrt{\Delta}}\left(\frac{-\alpha-(-\alpha)^N}{1+\alpha}-\frac{-\beta-(-\beta)^N}{1+\beta}\right) \\
&=\frac{1}{\sqrt{\Delta}}\left(\frac{-\alpha(1+\beta)-(1+\beta)(-\alpha)^N}{(1+\alpha)(1+\beta)}-\frac{-\beta(1+\alpha)-(1+\alpha)(-\beta)^N}{(1+\alpha)(1+\beta)}\right) \\
&=\frac{1}{N\sqrt{\Delta}}\left(-\alpha-\alpha\beta-(-\alpha-\alpha\beta)(-\alpha)^{N-1}+\beta+\alpha\beta-(\beta+\alpha\beta)(-\beta)^{N-1} \right) \quad \text{(by (\ref{formula:ab}))}\\
&=\frac{1}{N\sqrt{\Delta}} \left( -(\alpha-\beta)-(-\alpha-\alpha\beta)(-\alpha)^{N-1}-(\beta+\alpha\beta)(-\beta)^{N-1} \right) \\
&=\frac{1}{N\sqrt{\Delta}} \left( -\sqrt{\Delta}-(-\alpha-1)(-\alpha)^{N-1}-(\beta+1)(-\beta)^{N-1} \right) \\
&=\frac{1}{N} \left( -1-(-1)^N\frac{\alpha^N-\beta^N}{\sqrt{\Delta}}+(-1)^{N-1}\frac{\alpha^{N-1}-\beta^{N-1}}{\sqrt{\Delta}} \right) \\
&=\frac{(-1)^{N+1}(\mathcal B_N(N)+\mathcal B_{N-1}(N))-1}{N}.
\end{align*}

\item[(\ref{formula:4})]
\begin{align*}
\sum_{i=1}^{N-1}(-1)^i\mathcal P_i(N)&=\sum_{i=1}^{N-1}(-1)^i(\alpha^i+\beta^i) \\
&=\frac{-\alpha-(-\alpha)^N}{1+\alpha}+\frac{-\beta-(-\beta)^N}{1+\beta} \\
&=\frac{-\alpha(1+\beta)-(1+\beta)(-\alpha)^N}{(1+\alpha)(1+\beta)}+\frac{-\beta(1+\alpha)-(1+\alpha)(-\beta)^N}{(1+\alpha)(1+\beta)} \\
&=\frac{-\alpha-\alpha\beta-(-\alpha-\alpha\beta)(-\alpha)^{N-1}-\beta-\alpha\beta-(-\beta-\alpha\beta)(-\beta)^{N-1}}{N} \quad \text{(by (\ref{formula:ab}))}\\
&=\frac{1}{N} \left(-(\alpha+\beta)-2\alpha\beta-(-\alpha-\alpha\beta)(-\alpha)^{N-1}-(-\beta-\alpha\beta)(-\beta)^{N-1} \right) \\
&=\frac{1}{N} \left(-(N-2)-2-(-\alpha-1)(-\alpha)^{N-1}-(-\beta-1)(-\beta)^{N-1} \right) \\
&=\frac{1}{N}\left( -(-1)^N(\alpha^N+\beta^N)+(-1)^{N-1}(\alpha^{N-1}+\beta^{N-1})-N \right) \\
&=\frac{(-1)^{N+1}(\mathcal P_N(N)+\mathcal P_{N-1}(N))-N}{N}.
\end{align*}

\item[(\ref{formula:5})]
\begin{align*}
&\sum_{i=1}^{N-1}(-1)^ii\mathcal B_i(N) \\
&=\frac{1}{\sqrt{\Delta}}\sum_{i=1}^{N-1}(-1)^ii(\alpha^i-\beta^i) \\
&=\frac{1}{\sqrt{\Delta}}\left( \frac{1}{1+\alpha}\left( \sum_{i=1}^N(-\alpha)^i-N(-\alpha)^N \right)-\frac{1}{1+\beta}\left( \sum_{i=1}^N(-\beta)^i-N(-\beta)^N \right) \right) \quad \text{(by (\ref{formula:1}),(\ref{formula:2}))}\\
&=\frac{1}{\sqrt{\Delta}}\left( \frac{-\alpha-(-\alpha)^{N+1}}{(1+\alpha)^2}- \frac{-\beta-(-\beta)^{N+1}}{(1+\beta)^2}
-\frac{N(-\alpha)^N}{1+\alpha}+\frac{N(-\beta)^N}{1+\beta} \right) \\
&=\frac{1}{\sqrt{\Delta}}\times \left( \frac{-\alpha(1+\beta)^2-(1+\beta)^2(-\alpha)^{N+1}+\beta(1+\alpha)^2+(1+\alpha)^2(-\beta)^{N+1}}{(1+\alpha)^2(1+\beta)^2} \right.\\
&~\left. -\frac{N(1+\beta)(-\alpha)^N-N(1+\alpha)(-\beta)^N}{(1+\alpha)(1+\beta)} \right) \\
&=\frac{1}{\sqrt{\Delta}}\times \left(\frac{(-\alpha-\alpha\beta)(1+\beta)-(1+\beta)(-\alpha-\alpha\beta)(-\alpha)^N}{N^2} \right.\\
&~+\frac{(\beta+\alpha\beta)(1+\alpha)+(1+\alpha)(-\beta-\alpha\beta)(-\beta)^N}{N^2} \\
&~\left. -\frac{N(-\alpha-\alpha\beta)(-\alpha)^{N-1}-N(-\beta-\alpha\beta)(-\beta)^{N-1}}{N} \right) \quad \text{(by (\ref{formula:ab}))}\\
&=\frac{1}{N\sqrt{\Delta}}\times \left(\frac{-(\alpha+1)(1+\beta)+(1+\beta)(\alpha+1)(-\alpha)^N}{N} \right.\\
&~\left. +\frac{(\beta+1)(1+\alpha)-(1+\alpha)(\beta+1)(-\beta)^N}{N}
+N(\alpha+1)(-\alpha)^{N-1}-N(\beta+1)(-\beta)^{N-1} \right) \\
&=\frac{1}{N\sqrt{\Delta}}\times \left(\frac{-N+N(-\alpha)^N+N-N(-\beta)^N}{N} \right.\\
&~\left. +N(1+\alpha)(-\alpha)^{N-1}-N(1+\beta)(-\beta)^{N-1} \right) \quad \text{(by (\ref{formula:ab}))}\\
&=\frac{1}{N} \left( (-1)^N\frac{\alpha^N-\beta^N}{\sqrt{\Delta}}+N(-1)^{N-1}\frac{\alpha^{N-1}-\beta^{N-1}}{\sqrt{\Delta}}-N(-1)^N\frac{\alpha^N-\beta^N}{\sqrt{\Delta}} \right) \\
&=\frac{(-1)^N(\mathcal B_N(N)-N(\mathcal B_N(N)+\mathcal B_{N-1}(N)))}{N}.
\end{align*}

\item[(\ref{formula:6})]
\begin{align*}
&\sum_{i=1}^{N-1}(-1)^ii\mathcal P_i(N) \\
&=\sum_{i=1}^{N-1}(-1)^ii(\alpha^i+\beta^i) \\
&=\frac{1}{1+\alpha}\left( \sum_{i=1}^N(-\alpha)^i-N(-\alpha)^N \right)+\frac{1}{1+\beta}\left( \sum_{i=1}^N(-\beta)^i-N(-\beta)^N \right) \quad \text{(by (\ref{formula:1}),(\ref{formula:2}))}\\
&=\frac{-\alpha-(-\alpha)^{N+1}}{(1+\alpha)^2}+ \frac{-\beta-(-\beta)^{N+1}}{(1+\beta)^2}
-\frac{N(-\alpha)^N}{1+\alpha}-\frac{N(-\beta)^N}{1+\beta} \\
&=\frac{-\alpha(1+\beta)^2-(1+\beta)^2(-\alpha)^{N+1}-\beta(1+\alpha)^2-(1+\alpha)^2(-\beta)^{N+1}}{(1+\alpha)^2(1+\beta)^2} \\
&~-\frac{N(1+\beta)(-\alpha)^N+N(1+\alpha)(-\beta)^N}{(1+\alpha)(1+\beta)} \\
&=\frac{(-\alpha-\alpha\beta)(1+\beta)-(1+\beta)(-\alpha-\alpha\beta)(-\alpha)^N}{N^2} \\
&~+\frac{-(\beta+\alpha\beta)(1+\alpha)-(1+\alpha)(-\beta-\alpha\beta)(-\beta)^N}{N^2} \\
&-\frac{N(-\alpha-\alpha\beta)(-\alpha)^{N-1}+N(-\beta-\alpha\beta)(-\beta)^{N-1}}{N} \quad \text{(by (\ref{formula:ab}))}\\
&=\frac{-(\alpha+1)(1+\beta)+(1+\beta)(\alpha+1)(-\alpha)^N}{N^2}
+\frac{-(\beta+1)(1+\alpha)+(1+\alpha)(\beta+1)(-\beta)^N}{N^2} \\
&~-\frac{N(-\alpha-1)(-\alpha)^{N-1}+N(-\beta-1)(-\beta)^{N-1}}{N} \\
&=\frac{-2N+N(-1)^N(\alpha^N+\beta^N)}{N^2}
-\frac{N((-\alpha)^N-(-\alpha)^{N-1})+N((-\beta)^N-(-\beta)^{N-1})}{N} \quad \text{(by (\ref{formula:ab}))}\\
&=\frac{-2+(-1)^N(\alpha^N+\beta^N)}{N}-(-1)^N(\alpha^N+\beta^N)+(-1)^{N-1}(\alpha^{N-1}+\beta^{N-1}) \\
&=\frac{(-1)^N(\mathcal P_N(N)-N(\mathcal P_N(N)+\mathcal P_{N-1}(N)))-2}{N}.
\end{align*}

\end{itemize}

\end{proof}

From Theorem \ref{thm:r1}, we can obtain the following closed formula for the Kirchhoff index.

\begin{theorem}\label{thm:kirchhoff-r1}
\[
\Kf(G_N^{(1)})=\frac{N^2(1-\beta^N)}{\sqrt{\Delta}(1+\beta^N)}-1.
\]
\end{theorem}

\begin{proof}
\begin{align*}
\Kf(G_N^{(r)})&=\frac{1}{2}\sum_{x,y\in V}R(x,y) \\
&=\sum_{i=1}^{N-1}(N-i)R_{N,1}(i) \\
&=\sum_{i=1}^{N-1}(N-i)
\left(
(-1)^i \mathcal B_i(N) + \frac{1-\beta^N}{\sqrt{\Delta}(1+\beta^N)}
\left(
2-(-1)^i\mathcal P_i(N)
\right)
\right)\\
&=N\left(
\sum_{i=1}^{N-1}(-1)^i \mathcal B_i(N) + \frac{2(1-\beta^N)}{\sqrt{\Delta}(1+\beta^N)}\sum_{i=1}^{N-1}1
- \frac{(1-\beta^N)}{\sqrt{\Delta}(1+\beta^N)}\sum_{i=1}^{N-1} (-1)^i\mathcal P_i(N)
\right) \\
&~-\left(
\sum_{i=1}^{N-1}(-1)^i i\mathcal B_i(N) + \frac{2(1-\beta^N)}{\sqrt{\Delta}(1+\beta^N)}\sum_{i=1}^{N-1}i
- \frac{(1-\beta^N)}{\sqrt{\Delta}(1+\beta^N)}\sum_{i=1}^{N-1}(-1)^i i\mathcal P_i(N)
\right)\\
&=N\times \left(
\frac{(-1)^{N+1}(\mathcal B_N(N)+\mathcal B_{N-1}(N))-1}{N} + \frac{2(N-1)(1-\beta^N)}{\sqrt{\Delta}(1+\beta^N)} \right.\\
&~\left. -\frac{(1-\beta^N)}{\sqrt{\Delta}(1+\beta^N)}\left( \frac{(-1)^{N+1}(\mathcal P_N(N)+\mathcal P_{N-1}(N))-N}{N}
\right) \right)\\
&~-\frac{(-1)^N(\mathcal B_N(N)-N(\mathcal B_N(N)+\mathcal B_{N-1}(N)))}{N} - \frac{N(N-1)(1-\beta^N)}{\sqrt{\Delta}(1+\beta^N)} \\
&~+ \frac{(1-\beta^N)}{\sqrt{\Delta}(1+\beta^N)} \left( \frac{(-1)^N(\mathcal P_N(N)-N(\mathcal P_N(N)+\mathcal P_{N-1}(N)))-2}{N} \right) \quad \text{(by Lemma \ref{lem:PB41} (\ref{formula:3})--(\ref{formula:6}))}\\
&=-\frac{(-1)^N\mathcal B_N(N)+N}{N} + \frac{(1-\beta^N)}{\sqrt{\Delta}(1+\beta^N)}\frac{(-1)^N\mathcal P_N(N)+N^3-2}{N}\\
&=\frac{1}{N}\left(
(-1)^{N+1}\mathcal B_N(N)-N+\frac{(1-\beta^N)\left((-1)^N\mathcal P_N(N)+N^3-2\right)}{\sqrt{\Delta}(1+\beta^N)}\right) \\
&=\frac{1}{N}\left(
-N+\frac{\sqrt{\Delta}(1+\beta^N)\mathcal B_N(N)+(1-\beta^N)\left(\mathcal P_N(N)+N^3-2\right)}{\sqrt{\Delta}(1+\beta^N)}\right).
\intertext{Using $\mathcal B_N(N)=\frac{\beta^{-N}-\beta^N}{\sqrt{\Delta}}$ and $\mathcal P_N(N)=\beta^{-N}+\beta^N$,}
&=\frac{1}{N}\left(-N+\frac{(1+\beta^N)(\beta^{-N}-\beta^N)+(1-\beta^N)\left(\beta^{-N}+\beta^N+N^3-2\right)}{\sqrt{\Delta}(1+\beta^N)}\right) \\
&=\frac{1}{N}\left(-N+\frac{N^3(1-\beta^N)}{\sqrt{\Delta}(1+\beta^N)}\right) \\
&=\frac{N^2(1-\beta^N)}{\sqrt{\Delta}(1+\beta^N)}-1.
\end{align*}

\end{proof}

\subsubsection{Spanning Tree}

We give the following formula for the number of the spanning tree in terms of the Pisa number.

\begin{theorem} \label{thm:ST}
The number of the spanning tree of $G_N^{(1)}$ is
\[
t(G_N^{(1)})=\frac{\mathcal P_N(N)+2}{N^2}.
\]
\end{theorem}

\begin{proof}

Substituting $r=1$ into Equation \eqref{eq:mt} yields
\[
t(G_N^{(1)})=\frac{1}{N}\prod_{i=1}^{N-1}\left( N-2+2
\cos\left(\frac{2\pi i}{N}\right) \right).
\]

Furthermore,
\begin{align*}
\prod_{i=1}^{N-1}\left( N-2+2
\cos\left(\frac{2\pi i}{N}\right) \right)&=
\prod_{i=1}^{N-1}\left(-\left( -N+2-2
\cos\left(\frac{2\pi i}{N}\right) \right) \right) \\
&=(-1)^{N-1}\prod_{i=1}^{N-1}\left( -N+2-2
\cos\left(\frac{2\pi i}{N}\right) \right) \\
&=(-1)^{N-1}\frac{(-\alpha^N)+(-\alpha)^{-N}-2}{-\alpha-\alpha^{-1}-2} \\
&=(-1)^{N-1}\frac{(-\alpha)^N+(-\beta)^N-2}{-\alpha-\beta-2} \\
&=(-1)^{N-1}\frac{(-1)^N\mathcal P_N(N)-2}{-N} \\
&=\frac{\mathcal P_N(N)+2}{N}.
\end{align*}
Therefore, we obtain
\[
t(G_N^{(1)})=\frac{1}{N}\prod_{i=1}^{N-1}\left( N-2+2
\cos\left(\frac{2\pi i}{N}\right) \right)
=\frac{\mathcal P_N(N)+2}{N^2}.
\]
\end{proof}

From Theorem \ref{Kir} and Theorem \ref{thm:ST}, we obtain the following corollary.

\begin{corollary}
\[
t(G_N^{(1)};0,\ell)=
\frac{2(1+\beta^N)}{N^2\sqrt{\Delta}\beta^N}\left(1-\beta^N+(-1)^\ell(\beta^{N-\ell}-\beta^\ell)\right).
\]
\end{corollary}

\begin{proof}
\begin{align*}
t(G_N^{(1)};0,\ell)
&=R_{N,1}(\ell) \cdot t(G_N^{(1)}) \\
&=\frac{\mathcal P_N(N)+2}{N^2}\left( (-1)^\ell \mathcal B_\ell(N)+\frac{1-\beta^N}{\sqrt{\Delta}(1+\beta^N)}\left(2-(-1)^\ell\mathcal P_\ell(N)\right)\right).
\intertext{Using $\mathcal B_\ell(N)=\frac{\beta^{-\ell}-\beta^\ell}{\sqrt{\Delta}}$ and $\mathcal P_\ell(N)=\beta^{-\ell}+\beta^\ell$,}
&=\frac{\beta^{-N}+\beta^N+2}{N^2}\left((-1)^\ell \mathcal B_\ell(N)+\frac{1-\beta^N}{\sqrt{\Delta}(1+\beta^N)}\left(2-(-1)^\ell\mathcal P_\ell(N)\right) \right) \\
&=\frac{(1+\beta^N)^2}{N^2\beta^N}\left(\frac{(-1)^\ell(\beta^{-\ell}-\beta^\ell)}{\sqrt{\Delta}}+\frac{1-\beta^N}{\sqrt{\Delta}(1+\beta^N)}\left\{2-(-1)^\ell(\beta^{-\ell}+\beta^\ell)\right\}\right)\\
&=\frac{1+\beta^N}{N^2\sqrt{\Delta}\beta^N}
\left[(-1)^\ell(\beta^{-\ell}-\beta^\ell)(1+\beta^N)+(1-\beta^N)\left\{2-(-1)^\ell(\beta^{-\ell}+\beta^\ell)\right\}\right] \\
&=\frac{1+\beta^N}{N^2\sqrt{\Delta}\beta^N}\left[2(1-\beta^N)+2(-1)^\ell(\beta^{N-\ell}-\beta^\ell)\right] \\
&=\frac{2(1+\beta^N)}{N^2\sqrt{\Delta}\beta^N}\left(1-\beta^N+(-1)^\ell(\beta^{N-\ell}-\beta^\ell)\right).
\end{align*}

\end{proof}

\subsection{Case: $r=2$}

Next, we consider the case $r=2$.
In this case, $G_N^{(2)}$ is a non-trivial connected graph for $N \ge 7$.
On the other hand, for $N=5$, the degree of each vertex is $0$, resulting in a discrete graph.
Therefore, we assume $N \ge 7$.

We write $G_{N,2} = G_N^{(2)}$ and let $R_{N,2}(\ell)$ denote the effective resistance between two vertices at distance $\ell$ of $G_N^{(2)}$.

From Theorem \ref{mainthm1},
\[
R_{N,2}(\ell)
=
\frac{4}{N}
\sum_{m=1}^{(N-1)/2}
\frac{
1-\cos(\ell\theta_m)
}{
(N-4)+2\cos\theta_m+2\cos(2\theta_m)
},
\qquad
\theta_m=\frac{2\pi m}{N}.
\]
From the symmetry of the summand, we can write
\[
R_{N,2}(\ell)
=
\frac{2}{N}
\sum_{m=1}^{N-1}
\frac{
1-\cos(\ell\theta_m)
}{
(N-4)+2\cos\theta_m+2\cos(2\theta_m)
}.
\]

Let $t_m:=2\cos\theta_m$.
Then,
\[
(N-4)+2\cos\theta_m+2\cos(2\theta_m)
=
t_m^2+t_m+(N-6).
\]

Therefore, it can be seen that for $r=2$, the denominator of the effective resistance reduces to a quadratic polynomial.
Unlike the case of $r=1$, this leads to the emergence of a complex sequence structure.

The roots of the quadratic polynomial $t^2+t+(N-6)$ are given by
\[
\tau_\pm=\frac{-1\pm\sqrt{25-4N}}{2}.
\]
Setting
\[
s:=\sqrt{4N-25},
\qquad
\tau:=\frac{-1+is}{2},
\]
we have
\[
\tau_+=\tau,\qquad \tau_-=\overline{\tau},
\qquad
\tau_+-\tau_-=is.
\]

\begin{lemma}\label{lem:r2-reduction}
Let
\[
S(\tau)
:=
\frac1N
\sum_{m=1}^{N-1}
\frac{1-\cos(\ell\theta_m)}{t_m-\tau}.
\]
Then, it follows that
\[
R_{N,2}(\ell)
=
\frac{4}{s}\Im S(\tau),
\]
where $\Im(\cdot)$ denotes the imaginary part.
\end{lemma}

\begin{proof}
From partial fraction decomposition, we obtain
\[
\frac{1}{t^2+t+(N-6)}
=
\frac{1}{\tau_+-\tau_-}
\left(
\frac{1}{t-\tau_+}
-
\frac{1}{t-\tau_-}
\right).
\]
Therefore,
\[
R_{N,2}(\ell)
=
\frac{2}{\tau_+-\tau_-}
\left(S(\tau_+)-S(\tau_-)\right).
\]

Since each $t_m$ and $1-\cos(\ell\theta_m)$ is a real number, $S(\overline{\tau})=\overline{S(\tau)}$ holds.
Therefore, we obtain
\[
R_{N,2}(\ell)
=
\frac{2}{is}
\left(S(\tau)-\overline{S(\tau)}\right)
=
\frac{4}{s}\Im S(\tau).
\]
\end{proof}

\begin{remark}
In the case of $r=2$, although complex-valued functions and complex sequences are used in the course of the proof, the resulting effective resistance $R_{N,2}(\ell)$ is real-valued.
This is because the poles $\tau$ and $\overline{\tau}$ appearing in the partial fraction decomposition are complex conjugates of each other, causing the real parts to cancel out in the corresponding sum.
Consequently, the effective resistance is expressed in a form that extracts only the imaginary part.
\end{remark}

We estimate $S(\tau)$ using the following lemma.

\begin{lemma}\label{lem:root-of-unity}
Let $\rho\in\mathbb{C}$ be a complex number such that $\rho^N\neq 1$, and let $\omega=e^{2\pi i/N}$.
Then, for $k=0,1,\dots,N-1$,
\[
\frac1N
\sum_{m=0}^{N-1}
\frac{\omega^{mk}}{\omega^m-\rho}
=
\frac{\rho^{k-1}}{1-\rho^N}
\]
holds.

However, when $k=0$,
the right-hand side is interpreted as
\[
\frac{\rho^{N-1}}{1-\rho^N}.
\]
\end{lemma}

\begin{proof}
Performing a partial fraction decomposition using the identity
\[
\prod_{m=0}^{N-1}(z-\omega^m)=z^N-1,
\]
we obtain
\[
\frac{z^{k-1}}{z^N-1}
=
\frac1N
\sum_{m=0}^{N-1}
\frac{\omega^{mk}}{z-\omega^m}.
\]
Substituting $z=\rho$ into this expression and changing the sign of the denominator, the theorem follows.
\end{proof}

\begin{lemma}\label{lem:S-evaluation}
Let $\rho$ be a solution to
\[
\rho+\rho^{-1}=\tau,
\qquad |\rho|<1
\]
and let $d:=\rho-\rho^{-1}$.
Therefore,
\[
S(\tau)
=
\frac{
1+\rho^N-\rho^\ell-\rho^{N-\ell}
}{
d(1-\rho^N)
}
\]
holds.
\end{lemma}

\begin{proof}
Let $\omega=e^{2\pi i/N}$ and $z_m=\omega^m=e^{i\theta_m}$.
Then, $t_m=z_m+z_m^{-1}$.

Furthermore, from
\[
z^2-\tau z+1=(z-\rho)(z-\rho^{-1}),
\]
we obtain
\[
\frac{1}{t_m-\tau}
=
\frac{z_m}{(z_m-\rho)(z_m-\rho^{-1})}
=
\frac1d
\left(
\frac{\rho}{z_m-\rho}
-
\frac{\rho^{-1}}{z_m-\rho^{-1}}
\right).
\]

Moreover,
\[
1-\cos(\ell\theta_m)
=
\frac{2-z_m^\ell-z_m^{-\ell}}{2}.
\]

Note that since the term for $m=0$ is $1-\cos(\ell\theta_0)=0$, adding it to the sum does not change the value.
Therefore,
\[
S(\tau)
=
\frac{1}{2N}
\sum_{m=0}^{N-1}
(2-z_m^\ell-z_m^{-\ell})
\frac1d
\left(
\frac{\rho}{z_m-\rho}
-
\frac{\rho^{-1}}{z_m-\rho^{-1}}
\right)
\]
holds.

Applying Lemma \ref{lem:root-of-unity} for $k=0$, $\ell$, and $N-\ell$ and simplifying, we obtain
\[
S(\tau)
=
\frac{
1+\rho^N-\rho^\ell-\rho^{N-\ell}
}{
d(1-\rho^N)
}.
\]
\end{proof}

Here, we define the $r=2$ versions of the Pisa and Bejaia numbers by
\[
\mathcal P_k^{(2)}
:=
\rho^k+\rho^{-k},
\qquad
\mathcal B_k^{(2)}
:=
\frac{\rho^k-\rho^{-k}}{d}
\qquad (k\ge0).
\]

These can be regarded as complex analogues of the Bejaia and Pisa numbers appearing in the case $r=1$. 
Indeed, while the corresponding characteristic equation has real roots for $r=1$, complex roots arise for $r=2$, and thus the corresponding generalized Bejaia--Pisa sequences also become complex sequences.

Moreover, since $\rho+\rho^{-1}=\tau$, $\rho$ and $\rho^{-1}$ are the roots of $x^2-\tau x+1=0$.
Therefore, both $\mathcal P_k^{(2)}$ and $\mathcal B_k^{(2)}$ satisfy the complex Fibonacci-type recurrence relation $x_{k+2}=
\tau x_{k+1}-x_k$.
Therefore,
\[
\mathcal P_0^{(2)}=2,
\qquad
\mathcal P_1^{(2)}=\tau,
\]
\[
\mathcal B_0^{(2)}=0,
\qquad
\mathcal B_1^{(2)}=1.
\]

\begin{corollary}
The generalized Pisa sequence defined for $r=2$,
\[
\mathcal P_k^{(2)}
=
\rho^k+\rho^{-k},
\]
and the generalized Bejaia sequence,
\[
\mathcal B_k^{(2)}
=
\frac{\rho^k-\rho^{-k}}{d},
\]
both satisfy
\[
x_{k+2}
=
\tau x_{k+1}-x_k,
\]
where $\tau=\frac{-1+i\sqrt{4N-25}}{2}$.
\end{corollary}

\begin{proof}
From $\rho+\rho^{-1}=\tau$, $\rho$ and $\rho^{-1}$ are solutions of $x^2-\tau x+1=0$.
Therefore,
\[
\rho^{k+2}
=
\tau\rho^{k+1}-\rho^k,
\qquad
\rho^{-(k+2)}
=
\tau\rho^{-(k+1)}-\rho^{-k}
\]
holds.

Adding these two expressions yields
\[
\mathcal P_{k+2}^{(2)}
=
\tau\mathcal P_{k+1}^{(2)}
-
\mathcal P_k^{(2)}.
\]

Similarly,
subtracting these two expressions and dividing by $d=\rho-\rho^{-1}$ yields
\[
\mathcal B_{k+2}^{(2)}
=
\tau\mathcal B_{k+1}^{(2)}
-
\mathcal B_k^{(2)}.
\]
\end{proof}

Furthermore, let
\[
C_N
:=
\frac{
2+\mathcal P_N^{(2)}+d\mathcal B_N^{(2)}
}{
d\left(
2-\mathcal P_N^{(2)}-d\mathcal B_N^{(2)}
\right)
}.
\]
From the above, Theorem \ref{thm:r2-chair} follows.

\noindent
\textbf{Proof of Theorem \ref{thm:r2-chair}}

From Lemma \ref{lem:S-evaluation},
\[
S(\tau)
=
\frac{
1+\rho^N-\rho^\ell-\rho^{N-\ell}
}{
d(1-\rho^N)
}.
\]

Using $\rho^{N-\ell}=\rho^N\rho^{-\ell}$, we obtain
\[
1+\rho^N-\rho^\ell-\rho^{N-\ell}
=
(1+\rho^N)-(\rho^\ell+\rho^N\rho^{-\ell}).
\]

Moreover, from
\[
\rho^\ell
=
\frac{\mathcal P_\ell^{(2)}+d\mathcal B_\ell^{(2)}}{2},
\qquad
\rho^{-\ell}
=
\frac{\mathcal P_\ell^{(2)}-d\mathcal B_\ell^{(2)}}{2},
\]
we obtain
\[
\rho^\ell+\rho^N\rho^{-\ell}
=
\frac{
(1+\rho^N)\mathcal P_\ell^{(2)}
+
(1-\rho^N)d\mathcal B_\ell^{(2)}
}{2}.
\].

Therefore,
\[
S(\tau)
=
\frac{1+\rho^N}{2d(1-\rho^N)}
(2-\mathcal P_\ell^{(2)})
-
\frac12\mathcal B_\ell^{(2)}
\]
holds.

Let
\[
C_N'
:=
\frac{1+\rho^N}{d(1-\rho^N)}.
\]
Then,
\[
S(\tau)
=
C_N'
-
\frac{C_N'}{2}\mathcal P_\ell^{(2)}
-
\frac12\mathcal B_\ell^{(2)}.
\]

Moreover, from
\[
\rho^N
=
\frac{\mathcal P_N^{(2)}+d\mathcal B_N^{(2)}}{2},
\]
we obtain $C_N'=C_N$.

Therefore,
\[
S(\tau)
=
C_N
-
\frac{C_N}{2}\mathcal P_\ell^{(2)}
-
\frac12\mathcal B_\ell^{(2)}
\]
holds.

Finally, from Lemma \ref{lem:r2-reduction},
\[
R_{N,2}(\ell)
=
\frac4s\Im S(\tau),
\qquad s=\sqrt{4N-25}
\]
holds.
Therefore, we obtain
\[
R_{N,2}(\ell)
=
\frac{4}{\sqrt{4N-25}}
\Im\left(
C_N
-
\frac{C_N}{2}\mathcal P_\ell^{(2)}
-
\frac12\mathcal B_\ell^{(2)}
\right).
\]
\qed

Furthermore, from Theorem \ref{thm:NW} and Theorem \ref{thm:r2-chair},
we give the following closed formula for the average hitting times of $G_N^{(2)}$.

\begin{corollary}
\[
h(0, \ell)
=\frac{2N(N-5)}{\sqrt{4N-25}}
\Im\left(
C_N - \frac{C_N}{2}\mathcal P_\ell^{(2)} - \frac12\mathcal B_\ell^{(2)}
\right).
\]
\end{corollary}

\begin{proof}

\begin{align*}
h(0, \ell)
&=|E|R_{N,1}(\ell) \\
&=\frac{2N(N-5)}{\sqrt{4N-25}}
\Im\left(
C_N - \frac{C_N}{2}\mathcal P_\ell^{(2)} - \frac12\mathcal B_\ell^{(2)}
\right).
\end{align*}

\end{proof}

\subsubsection{Kirchhoff Index}

Before calculating the Kirchhoff index, we prepare some formulas.

\begin{lemma}\label{lem:PB42}
\begin{enumerate}
    \item\label{formula2:rho} $\sum_{i=1}^{N-1}i\rho^i=\frac{1}{1-\rho}\left( \sum_{i=1}^N\rho^i-N\rho^N \right)$

    \item\label{formula2:rho2} $\sum_{i=1}^{N-1}i\rho^{-i}=\frac{1}{1-\rho^{-1}}\left( \sum_{i=1}^N\rho^{-i}-N\rho^{-N} \right)$

    \item\label{formula2:1} $\sum_{i=1}^{N-1}\mathcal B_i^{(2)}=\frac{\mathcal B_1^{(2)}-\mathcal B_N^{(2)}+\mathcal B_{N-1}^{(2)}}{2-\tau}$

    \item\label{formula2:2} $\sum_{i=1}^{N-1}\mathcal P_i^{(2)}=\frac{\mathcal P_1^{(2)}-\mathcal P_N^{(2)}+\mathcal P_{N-1}^{(2)}-2}{2-\tau}$

    \item\label{formula2:3} $\sum_{i=1}^{N-1}i\mathcal B_i^{(2)}=\frac{\mathcal B_N^{(2)}-N(\mathcal B_N^{(2)}-\mathcal B_{N-1}^{(2)})}{2-\tau}$

    \item\label{formula2:4} $\sum_{i=1}^{N-1}i\mathcal P_i^{(2)}=\frac{\mathcal P_N^{(2)}-N(\mathcal P_N^{(2)}-\mathcal P_{N-1}^{(2)})-2}{2-\tau}$

    \item\label{formula2:5} $C_N\left(\mathcal P_N^{(2)}-2\right)+\mathcal B_N^{(2)}=0$.
\end{enumerate}

\end{lemma}

\begin{proof}

\begin{itemize}
    
\item[(\ref{formula2:rho})]

\begin{align*}
\sum_{i=1}^{N-1}i\rho^i-\rho\sum_{i=1}^{N-1}i\rho^i
&=\rho+\rho^2+\cdots+\rho^{N-1}-(N-1)\rho^N \\
&=\sum_{i=1}^N\rho^i-N\rho^N.
\end{align*}
From the above, we obtain
\[
\sum_{i=1}^{N-1}i\rho^i=\frac{1}{1-\rho}\left( \sum_{i=1}^N\rho^i-N\rho^N \right).
\]

\item[(\ref{formula2:rho2})]

\begin{align*}
\sum_{i=1}^{N-1}i\rho^{-i}-\rho^{-1}\sum_{i=1}^{N-1}i\rho^{-i}
&=\rho^{-1}+\rho^{-2}+\cdots+\rho^{-(N-1)}-(N-1)\rho^{-N} \\
&=\sum_{i=1}^N\rho^{-i}-N\rho^{-N}.
\end{align*}
From the above, we obtain
\[
\sum_{i=1}^{N-1}i\rho^{-i}=\frac{1}{1-\rho^{-1}}\left( \sum_{i=1}^N\rho^{-i}-N\rho^{-N} \right).
\]

\item[(\ref{formula2:1})]

\begin{align*}
\sum_{i=1}^{N-1}\mathcal B_i^{(2)}&=\sum_{i=1}^{N-1}\frac{\rho^i-\rho^{-i}}{d} \\
&=\frac{1}{d}(\sum_{i=1}^{N-1}\rho^i - \sum_{i=1}^{N-1}\rho^{-i}) \\
&=\frac{1}{d}\left( \frac{\rho(1-\rho^{N-1})}{1-\rho} - \frac{\rho^{-1}(1-\rho^{-(N-1)})}{1-\rho^{-1}} \right) \\
&=\frac{1}{d}\left( \frac{\rho(1-\rho^{-1})(1-\rho^{N-1})}{(1-\rho)(1-\rho^{-1})} - \frac{\rho^{-1}(1-\rho)(1-\rho^{-(N-1)})}{(1-\rho)(1-\rho^{-1})} \right) \\
&=\frac{1}{d(1-\rho)(1-\rho^{-1})}\left( (\rho-1)(1-\rho^{N-1}) - (\rho^{-1}-1)(1-\rho^{-(N-1)}) \right) \\
&=\frac{1}{d(2-\rho-\rho^{-1})}\left( \rho(1-\rho^{N-1})+\rho^{N-1} - \rho^{-1}(1-\rho^{-(N-1)})-\rho^{-(N-1)} \right) \\
&=\frac{1}{2-\tau}\left( \frac{\rho-\rho^{-1}}{d} -\frac{\rho^N-\rho^{-N}}{d}+\frac{\rho^{N-1}-\rho^{-(N-1)}}{d} \right) \\
&=\frac{\mathcal B_1^{(2)}-\mathcal B_N^{(2)}+\mathcal B_{N-1}^{(2)}}{2-\tau}.
\end{align*}

\item[(\ref{formula2:2})]

\begin{align*}
\sum_{i=1}^{N-1}\mathcal P_i^{(2)}&=\sum_{i=1}^{N-1}(\rho^i+\rho^{-i}) \\
&=\sum_{i=1}^{N-1}\rho^i + \sum_{i=1}^{N-1}\rho^{-i} \\
&=\frac{\rho(1-\rho^{N-1})}{1-\rho} + \frac{\rho^{-1}(1-\rho^{-(N-1)})}{1-\rho^{-1}} \\
&=\frac{\rho(1-\rho^{-1})(1-\rho^{N-1})}{(1-\rho)(1-\rho^{-1})} + \frac{\rho^{-1}(1-\rho)(1-\rho^{-(N-1)})}{(1-\rho)(1-\rho^{-1})} \\
&=\frac{1}{(1-\rho)(1-\rho^{-1})}\left( (\rho-1)(1-\rho^{N-1}) + (\rho^{-1}-1)(1-\rho^{-(N-1)}) \right) \\
&=\frac{1}{2-\rho-\rho^{-1}}\left( \rho(1-\rho^{N-1})+\rho^{N-1} + \rho^{-1}(1-\rho^{-(N-1)})+\rho^{-(N-1)} -2\right) \\
&=\frac{1}{2-\tau}\left( \rho+\rho^{-1} -(\rho^N+\rho^{-N}) +\rho^{N-1}+\rho^{-(N-1)} -2\right) \\
&=\frac{\mathcal P_1^{(2)}-\mathcal P_N^{(2)}+\mathcal P_{N-1}^{(2)}-2}{2-\tau}.
\end{align*}

\item[(\ref{formula2:3})]

\begin{align*}
\sum_{i=1}^{N-1}i\mathcal B_i^{(2)}&=\sum_{i=1}^{N-1}i\frac{\rho^i-\rho^{-i}}{d} \\
&=\frac{1}{d}\left( \sum_{i=1}^{N-1}i\rho^i-\sum_{i=1}^{N-1}i\rho^{-i} \right) \\
&=\frac{1}{d}\left( \frac{1}{1-\rho}\left( \sum_{i=1}^N\rho^i-N\rho^N \right)-\frac{1}{1-\rho^{-1}}\left( \sum_{i=1}^N\rho^{-i}-N\rho^{-N} \right) \right) \quad \text{(by (\ref{formula2:rho}),(\ref{formula2:rho2}))}\\
&=\frac{1}{d}\left( \frac{1}{1-\rho}\left( \frac{\rho(1-\rho^N)}{1-\rho}-N\rho^N \right)-\frac{1}{1-\rho^{-1}}\left( \frac{\rho^{-1}(1-\rho^{-N})}{1-\rho^{-1}}-N\rho^{-N} \right) \right) \\
&=\frac{1}{d}\left( \frac{\rho(1-\rho^N)}{(1-\rho)^2}-\frac{N\rho^N}{1-\rho}
-\frac{\rho^{-1}(1-\rho^{-N})}{(1-\rho^{-1})^2}+\frac{N\rho^{-N}}{1-\rho^{-1}} \right) \\
&=\frac{1}{d}\left( \frac{\rho(1-\rho^{-1})^2(1-\rho^N)-\rho^{-1}(1-\rho)^2(1-\rho^{-N})}{(1-\rho)^2(1-\rho^{-1})^2}
-\frac{N(1-\rho^{-1})\rho^N-N(1-\rho)\rho^{-N}}{(1-\rho)(1-\rho^{-1})} \right) \\
&=\frac{1}{d}\times\left( \frac{-(2-\rho-\rho^{-1})(1-\rho^N)+(2-\rho-\rho^{-1})(1-\rho^{-N})}{(2-\rho-\rho^{-1})^2} \right.\\
&~\left. -\frac{N(\rho-1)\rho^{N-1}-N(\rho^{-1}-1)\rho^{-(N-1)}}{2-\rho-\rho^{-1}} \right) \\
&=\frac{1}{d}\left( \frac{-(1-\rho^N)+(1-\rho^{-N})}{2-\rho-\rho^{-1}}-\frac{N(\rho-1)\rho^{N-1}-N(\rho^{-1}-1)\rho^{-(N-1)}}{2-\rho-\rho^{-1}} \right) \\
&=\frac{1}{2-\tau}\left( \frac{\rho^N-\rho^{-N}}{d}-N\left(\frac{\rho^N-\rho^{-N}}{d}-\frac{\rho^{N-1}-\rho^{-(N-1)}}{d}\right) \right) \\
&=\frac{\mathcal B_N^{(2)}-N(\mathcal B_N^{(2)}-\mathcal B_{N-1}^{(2)})}{2-\tau}.
\end{align*}

\item[(\ref{formula2:4})]

\begin{align*}
\sum_{i=1}^{N-1}i\mathcal P_i^{(2)}&=\sum_{i=1}^{N-1}i(\rho^i+\rho^{-i}) \\
&=\sum_{i=1}^{N-1}i\rho^i+\sum_{i=1}^{N-1}i\rho^{-i} \\
&=\frac{1}{1-\rho}\left( \sum_{i=1}^N\rho^i-N\rho^N \right)
+\frac{1}{1-\rho^{-1}}\left( \sum_{i=1}^N\rho^{-i}-N\rho^{-N} \right) \quad \text{(by (\ref{formula2:rho}),(\ref{formula2:rho2}))}\\
&=\frac{1}{1-\rho}\left( \frac{\rho(1-\rho^N)}{1-\rho}-N\rho^N \right)
+\frac{1}{1-\rho^{-1}}\left( \frac{\rho^{-1}(1-\rho^{-N})}{1-\rho^{-1}}-N\rho^{-N} \right) \\
&=\frac{\rho(1-\rho^N)}{(1-\rho)^2}-\frac{N\rho^N}{1-\rho}
+\frac{\rho^{-1}(1-\rho^{-N})}{(1-\rho^{-1})^2}-\frac{N\rho^{-N}}{1-\rho^{-1}} \\
&=\frac{\rho(1-\rho^{-1})^2(1-\rho^N)+\rho^{-1}(1-\rho)^2(1-\rho^{-N})}{(1-\rho)^2(1-\rho^{-1})^2}
-\frac{N(1-\rho^{-1})\rho^N+N(1-\rho)\rho^{-N}}{(1-\rho)(1-\rho^{-1})} \\
&=\frac{-(2-\rho-\rho^{-1})(1-\rho^N)-(2-\rho-\rho^{-1})(1-\rho^{-N})}{(2-\rho-\rho^{-1})^2} \\
&~+\frac{N(1-\rho)\rho^{N-1}+N(1-\rho^{-1})\rho^{-(N-1)}}{2-\rho-\rho^{-1}} \\
&=\frac{\rho^N+\rho^{-N}-N(\rho^N+\rho^{-N}-(\rho^{N-1}+\rho^{-(N-1)}))-2}{2-\tau} \\
&=\frac{\mathcal P_N^{(2)}-N(\mathcal P_N^{(2)}-\mathcal P_{N-1}^{(2)})-2}{2-\tau}.
\end{align*}

\item[(\ref{formula2:5})]

From the definition of $C_N$, it holds
\begin{align*}
C_N\left(\mathcal P_N^{(2)}-2\right)+\mathcal B_N^{(2)}
&=\frac{\left(2+\mathcal P_N^{(2)}+d\mathcal B_N^{(2)}\right)\left(\mathcal P_N^{(2)}-2\right)}{d\left(2-\mathcal P_N^{(2)}-d\mathcal B_N^{(2)}\right)}+\mathcal B_N^{(2)} \\
&=\frac{\left(2+\mathcal P_N^{(2)}+d\mathcal B_N^{(2)}\right)\left(\mathcal P_N^{(2)}-2\right)+d\mathcal B_N^{(2)}\left(2-\mathcal P_N^{(2)}-d\mathcal B_N^{(2)}\right)}{d\left(2-\mathcal P_N^{(2)}-d\mathcal B_N^{(2)}\right)}\\
&=\frac{\left(\mathcal P_N^{(2)}\right)^2-4-d^2\left(\mathcal B_N^{(2)}\right)^2}{d\left(2-\mathcal P_N^{(2)}-d\mathcal B_N^{(2)}\right)} \\
&=\frac{\left(\rho^N+\rho^{-N}\right)^2-4-d^2\left(\frac{\rho^N-\rho^{-N}}{d}\right)^2}{d\left(2-\mathcal P_N^{(2)}-d\mathcal B_N^{(2)}\right)} \\
&=\frac{\left(\rho^N+\rho^{-N}\right)^2-4-\left(\rho^N-\rho^{-N}\right)^2}{d\left(2-\mathcal P_N^{(2)}-d\mathcal B_N^{(2)}\right)}\\
&=0.
\end{align*}

\end{itemize}

\end{proof}

From Theorem \ref{thm:r2-chair}, we can obtain the following closed formula for the Kirchhoff index.

\begin{theorem}\label{thm:kirchhoff-r2}
\[
\Kf(G_{N,2})=\frac{2N^2}{\sqrt{4N-25}}\Im(C_N)-1.
\]
\end{theorem}

\begin{proof}

\begin{align*}
\Kf(G_{N,2})&=\frac{1}{2}\sum_{x,y\in V}R(x,y) \\
&=\sum_{i=1}^{N-1}(N-i)R_{N,2}(i) \\
&=\sum_{i=1}^{N-1}(N-i)\frac{4}{\sqrt{4N-25}}
\Im\left(
C_N - \frac{C_N}{2}\mathcal P_i^{(2)} - \frac12\mathcal B_i^{(2)}
\right)\\
&=\frac{4}{\sqrt{4N-25}}\times \\
&~\Im\left(NC_N\sum_{i=1}^{N-1}1 - \frac{NC_N}{2}\sum_{i=1}^{N-1}\mathcal P_i^{(2)} - \frac{N}{2}\sum_{i=1}^{N-1}\mathcal B_i^{(2)}
-C_N\sum_{i=1}^{N-1}i + \frac{C_N}{2}\sum_{i=1}^{N-1}i\mathcal P_i^{(2)} + \frac12\sum_{i=1}^{N-1}i\mathcal B_i^{(2)}
\right)\\
&=\frac{4}{\sqrt{4N-25}}\Im\left( N(N-1)C_N - \frac{NC_N}{2}\frac{\mathcal P_1^{(2)}-\mathcal P_N^{(2)}+\mathcal P_{N-1}^{(2)}-2}{2-\tau} \right. \\
&~- \frac{N}{2}\frac{\mathcal B_1^{(2)}-\mathcal B_N^{(2)}+\mathcal B_{N-1}^{(2)}}{2-\tau}-\frac{N(N-1)C_N}{2} + \frac{C_N}{2}\frac{\mathcal P_N^{(2)}-N(\mathcal P_N^{(2)}-\mathcal P_{N-1}^{(2)})-2}{2-\tau} \\
&~\left.+ \frac12\frac{\mathcal B_N^{(2)}-N(\mathcal B_N^{(2)}-\mathcal B_{N-1}^{(2)})}{2-\tau} \right) \quad \text{(by Lemma \ref{lem:PB42} (\ref{formula2:1})--(\ref{formula2:4})}\\
&=\frac{4}{\sqrt{4N-25}}
\Im \left( \frac{C_N}{2(2-\tau)}\left(-N\mathcal P_1^{(2)}+N\mathcal P_N^{(2)}-N\mathcal P_{N-1}^{(2)}+2N+\mathcal P_N^{(2)}-N(\mathcal P_N^{(2)}-\mathcal P_{N-1}^{(2)})-2 \right) \right.\\
&~\left.+\frac{N(N-1)C_N}{2} 
+ \frac{1}{2(2-\tau)}\left(-N\mathcal B_1^{(2)}+N\mathcal B_N^{(2)}-N\mathcal B_{N-1}^{(2)}+\mathcal B_N^{(2)}-N(\mathcal B_N^{(2)}-\mathcal B_{N-1}^{(2)})\right) \right)\\
&=\frac{4}{\sqrt{4N-25}}
\Im \left( \frac{C_N\left(-N\mathcal P_1^{(2)}+\mathcal P_N^{(2)}+2N-2 \right)}{2(2-\tau)}+\frac{N(N-1)(2-\tau)C_N}{2(2-\tau)} 
+ \frac{-N\mathcal B_1^{(2)}+\mathcal B_N^{(2)}}{2(2-\tau)} \right)\\
&=\frac{2}{\sqrt{4N-25}}
\Im \left(\frac{C_N\left(-N^2\mathcal P_1^{(2)}+\mathcal P_N^{(2)}+2N^2-2 \right)-N\mathcal B_1^{(2)}+\mathcal B_N^{(2)}}{2-\tau}\right) \\
&=\frac{2}{\sqrt{4N-25}}\Im \left(\frac{C_N\left(N^2(2-\tau)+\mathcal P_N^{(2)}-2 \right)-N+\mathcal B_N^{(2)}}{2-\tau} \right)\\
&=\frac{2}{\sqrt{4N-25}}\Im \left( N^2C_N+\frac{C_N\left(\mathcal P_N^{(2)}-2\right)+\mathcal B_N^{(2)}-N}{2-\tau}\right) \\
&=\frac{2}{\sqrt{4N-25}}\Im \left( N^2C_N-\frac{N}{2-\tau}\right) \quad \text{(by Lemma \ref{lem:PB42} (\ref{formula2:5}))}\\
&=\frac{2}{\sqrt{4N-25}}\Im \left(N^2C_N-\frac{2N}{5-i\sqrt{4N-25}}\right) \\
&=\frac{2}{\sqrt{4N-25}}\Im \left(N^2C_N-\frac{5+i\sqrt{4N-25}}{2}\right) \\
&=\frac{2N^2}{\sqrt{4N-25}}\Im \left(C_N-\frac{5+i\sqrt{4N-25}}{2N^2}\right) \\
&=\frac{2N^2}{\sqrt{4N-25}}\Im(C_N)-1.
\end{align*}

\end{proof}

\subsubsection{Spanning Tree}

We give the following formula for the number of the spanning tree in terms of the generalized Pisa number.

\begin{theorem} \label{thm:ST2}
The number of the spanning tree of $G_N^{(2)}$ is
\[
t(G_N^{(2)})=\frac{|\mathcal P_N^{(2)}-2|^2}{N^2}.
\]
\end{theorem}

\begin{proof}

Substituting $r=2$ into Equation \eqref{eq:mt} yields
\[
t(G_N^{(2)})=\frac{1}{N}\prod_{i=1}^{N-1}\left( N-4+2
\cos\left(\frac{2\pi i}{N}\right)
+2\cos\left(\frac{4\pi i}{N}\right)\right).
\]

Furthermore,
\begin{align*}
&\prod_{i=1}^{N-1}\left( N-4+2
\cos\left(\frac{2\pi i}{N}\right)
+2\cos\left(\frac{4\pi i}{N}\right)\right) \\
&=\prod_{i=1}^{N-1}\left( N-6+2
\cos\left(\frac{2\pi i}{N}\right)
+4\cos^2\left(\frac{2\pi i}{N}\right)\right) \\
&=(-1)^{N-1}\prod_{i=1}^{N-1}\left(\tau-2
\cos\left(\frac{2\pi i}{N}\right) \right)
(-1)^{N-1}\prod_{i=1}^{N-1}\left(\overline{\tau}-2
\cos\left(\frac{2\pi i}{N}\right) \right)\\
&=\prod_{i=1}^{N-1}\left(\tau-2
\cos\left(\frac{2\pi i}{N}\right) \right)
\prod_{i=1}^{N-1}\left(\overline{\tau}-2
\cos\left(\frac{2\pi i}{N}\right) \right)\\
&=\frac{(\rho^N+\rho^{-N}-2)(\overline{\rho}^N+\overline{\rho}^{-N}-2)}{(\rho+\rho^{-1}-2)(\overline{\rho}+\overline{\rho}^{-1}-2)} \\
&=\frac{(\mathcal P_N^{(2)}-2)(\overline{\mathcal P_N^{(2)}}-2)}{(\tau-2)(\overline{\tau}-2)} \\
&=\frac{|\mathcal P_N^{(2)}-2|^2}{N}.
\end{align*}
Therefore, we obtain
\[
t(G_N^{(1)})=\frac{1}{N}\prod_{i=1}^{N-1}\left( N-4+2
\cos\left(\frac{2\pi i}{N}\right)
+2\cos\left(\frac{4\pi i}{N}\right)\right)=\frac{|\mathcal P_N^{(2)}-2|^2}{N^2}.
\]
\end{proof}

From Theorem \ref{Kir} and Theorem \ref{thm:ST2}, we obtain the following corollary.

\begin{corollary}
\[
t(G_N^{(2)};0,\ell)=\frac{4|\mathcal P_N^{(2)}-2|^2}{N^2\sqrt{4N-25}}
\Im\left(
C_N
-
\frac{C_N}{2}\mathcal P_\ell^{(2)}
-
\frac12\mathcal B_\ell^{(2)}
\right).
\]
\end{corollary}

\begin{proof}
\begin{align*}
t(G_N^{(2)};0,\ell)&=R_{N,2}(\ell) \cdot t(G_N^{(2)}) \\
&=\frac{4}{\sqrt{4N-25}}
\Im\left(
C_N
-
\frac{C_N}{2}\mathcal P_\ell^{(2)}
-
\frac12\mathcal B_\ell^{(2)}
\right)\cdot \frac{|\mathcal P_N^{(2)}-2|^2}{N^2} \\
&=\frac{4|\mathcal P_N^{(2)}-2|^2}{N^2\sqrt{4N-25}}
\Im\left(
C_N
-
\frac{C_N}{2}\mathcal P_\ell^{(2)}
-
\frac12\mathcal B_\ell^{(2)}
\right).
\end{align*}
\end{proof}

\subsection{Numerical Verification}

Finally, we verify the numerical agreement between the closed form obtained in Theorem \ref{thm:r2-chair} and the finite sum representation given by Theorem \ref{mainthm1}.

For example, by computing the values for $N=7$, we obtain the following table.
\[
\begin{array}{c|c|c|c}
\ell
&
\text{Finite Sum Representation}
&
\text{Closed Form}
&
\text{Deviation}
\\ \hline
1 & 0.7777777778 & 0.7777777778 & <10^{-12} \\
2 & 0.8888888889 & 0.8888888889 & <10^{-12} \\
3 & 0.7777777778 & 0.7777777778 & <10^{-12}
\end{array}
\]

Similar agreement is verified for other odd integers $N$.
This numerically confirms that the generalized Bejaia--Pisa-type closed form derived for $r=2$ coincides with the finite sum representation.

\section{Conclusion}

In this paper, we studied the effective resistance of the graph $G_N^{(r)}$ obtained by deleting the edges corresponding to the cyclic distances $\{\pm1,\pm2,\dots,\pm r\}$ from the complete graph $K_N$.
Using the cyclic symmetry of the graph to diagonalize the Laplacian matrix via the discrete Fourier basis, we derived a finite trigonometric sum representation for the effective resistance between two vertices at distance $\ell$.
Furthermore, we obtained an explicit formula for the average hitting time.

After that, we analyzed the cases $r=1$ and $r=2$ in detail.
For $r=1$, if $N$ is odd, the graph $G_N^{(1)}=K_N\setminus C_N$ coincides with the one studied by Chair \cite{Chair}, and the effective resistance is expressed in a closed form in terms of Bejaia and Pisa numbers.
By rederiving this result from the perspective of Fourier analysis and root-of-unity expansions, we showed that these Fibonacci-type sequences naturally arise from the spectral structure of the Laplacian.

The main novelty of this paper lies in the complex sequence structure that arises for $r=2$.
In this case, the denominator of the effective resistance reduces to a quadratic polynomial, and we showed that by using its complex roots, the generalized Bejaia--Pisa-type sequences $\mathcal P_k^{(2)}$ and $\mathcal B_k^{(2)}$, which correspond to the Bejaia and Pisa numbers for $r=1$, are naturally derived.

Specifically, these sequences are described in terms of the complex number $\rho$ determined by $\rho+\rho^{-1}=\frac{-1+i\sqrt{4N-25}}{2}$ and satisfy the complex Fibonacci-type recurrence relation $x_{k+2}=\tau x_{k+1}-x_k$.
Therefore, it has become clear that an algebraic and sequential structure lies behind the effective resistance for the circulant distance deletion of the complete graph.

Specifically, while a real Fibonacci-type structure corresponding to real roots appears for $r=1$, an oscillatory complex Fibonacci-type structure corresponding to complex roots arises for $r=2$.
This suggests that the structure of the effective resistance itself changes significantly depending on the deleted distance classes.

These results suggest the possibility that generalized sequences corresponding to higher-order algebraic equations may arise for more general $r$ as well.
The existence of closed forms for $r \ge3$, the corresponding recurrence structures, and the extension to weighted circulant graphs and other vertex-transitive graphs remain interesting problems for future research.

The results of this paper show that the effective resistance of circulant graphs, obtained by deleting distance classes from the complete graph, is not merely a trigonometric sum but is deeply connected to Fibonacci-type sequences and generalized sequences through spectral decomposition.
This suggests the potential existence of novel relationships among effective resistance, random walks, spectral graph theory, and integer sequences.

\section*{Acknowledgments}

This study was supported by JSPS KAKENHI Grant Number JP25K23339 (to Y. T.). 

\appendix
\section{Even number of vertices} \label{app:evenN}

In this appendix, we briefly discuss the case where the number of vertices $N$ is even.
We set $V=\mathbb Z/N\mathbb Z$ and consider an integer $r$ satisfying $1\le r\le \frac{N}{2}-1$.


Similarly to the odd-vertex case, we define the cyclic distance by
\[
d(x,y)=\min\{|x-y|,\,N-|x-y|\}.
\]
However, when $N$ is even, the maximum value of the distance is $h=N/2$.
In this case, the cardinality of the distance class $D_k=\{x\in V\mid d(0,x)=k\}$ is given by
\[
|D_0|=|D_h|=1,
\qquad
|D_k|=2
\quad
(1\le k\le h-1).
\]

Therefore, only the distance class corresponding to the maximum distance $h=N/2$ consists of a single vertex.
This constitutes the essential difference from the odd-vertex case.


Since the graph $G_N^{(r)}$ is a circulant graph, as in the case of an odd number of vertices, its Laplacian matrix can be diagonalized via the discrete Fourier basis
\[
\varphi_m(x)
=
\frac1{\sqrt N}
e^{2\pi imx/N}
\qquad
(m=0,1,\dots,N-1).
\]

The corresponding eigenvalues are given, similarly to the odd-vertex case, by
\[
\lambda_m
=
(N-2r)
+
2\sum_{j=1}^r
\cos\left(
\frac{2\pi mj}{N}
\right).
\]


Therefore, by using Wu's formula, we obtain the formula for the effective resistance $R(0,\ell)$,
\[
R(0,\ell)
=
\frac{2}{N}
\sum_{m=1}^{N-1}
\frac{
1-\cos\left(\frac{2\pi m\ell}{N}\right)
}{
(N-2r)
+
2\sum_{j=1}^r
\cos\left(
\frac{2\pi mj}{N}
\right)
}.
\]

In the odd-vertex case, utilizing the symmetry $m \longleftrightarrow N-m$
enabled us to reduce the number of terms in the sum by half.

However, in the even-vertex case, the term corresponding to $m=\frac{N}{2}$ is self-symmetric, so we need to treat this term separately.


\begin{theorem}
Let $N$ be an even number and $r$ be an integer such that $1\le r \le \frac N2-1$.

Then, the effective resistance between two vertices at distance $\ell$ is
\[
R(0,\ell)
=
\frac{4}{N}
\sum_{m=1}^{N/2-1}
\frac{
1-\cos\left(\frac{2\pi m\ell}{N}\right)
}{
(N-2r)
+
2\sum_{j=1}^r
\cos\left(
\frac{2\pi mj}{N}
\right)
}
+
\frac{2}{N}
\frac{
1-\cos(\pi\ell)
}{
(N-2r)
+
2\sum_{j=1}^r
\cos(\pi j)
}.
\]
\end{theorem}

\begin{proof}
We consider the formula
\[
R(0,\ell)
=
\frac{2}{N}
\sum_{m=1}^{N-1}
\frac{
1-\cos\left(\frac{2\pi m\ell}{N}\right)
}{
(N-2r)
+
2\sum_{j=1}^r
\cos\left(
\frac{2\pi mj}{N}
\right)
}.
\]

As in the odd-vertex case, the summand is invariant under $m\longleftrightarrow N-m$.
However, in the even-vertex case, $m=\frac{N}{2}$ is self-symmetric, so this term remains alone.

Therefore, by pairing the terms for $m=1,\dots,\frac{N}{2}-1$, we obtain the theorem.
\end{proof}


\begin{remark}
For the even-vertex case, the basic structure of the effective resistance is the same as the odd-vertex case except for the appearance of a correction term corresponding to $m = \frac{N}{2}$. 

Specifically, we expect that the derivation of closed forms for $r=1$ and $r=2$ can be derived similarly by adding this correction term.
\end{remark}


\end{document}